\documentclass[12pt,twoside]{amsart}
\usepackage{amsthm,amsfonts,amssymb,amscd,euscript,enumerate}
\usepackage[all]{xy}

\newcommand{\C}{{\mathbb C}}

\newcommand{\Q}{{\mathbb Q}}
\newcommand{\Z}{{\mathbb Z}}
\newcommand{\N}{{\mathbb N}}
\newcommand{\R}{{\mathbb R}}

\newcommand{\dg}{\operatorname{dg}}

\newcommand{\dinv}{\operatorname{dinv}}

\newcommand{\Par}{\operatorname{Par}}

\newcommand{\sgn}{\operatorname{sgn}}

\theoremstyle{plain}
\newtheorem{thm}{Theorem}[section]

\newtheorem{lem}[thm]{Lemma}
\newtheorem{cor}[thm]{Corollary}

\newtheorem{prop}[thm]{Proposition}

\newtheorem{conj}[thm]{Conjecture}

\theoremstyle{definition}
\newtheorem{defn}[thm]{Definition}
\newtheorem{defnprop}[thm]{Definition-Proposition}

\newtheorem{exmp}[thm]{Example}

\numberwithin{equation}{section}

\theoremstyle{remark}

\newcommand{\dqbin}[3]{\displaystyle\genfrac{[}{]}{0pt}{}{#1}{#2}_{#3}}
\def\C{{\mathbb C}}
\def\D{{\mathfrak D}}
\def\N{{\mathbb N}}
\def\Q{{\mathbb Q}}
\def\Z{{\mathbb Z}}
\def\w{{\mathrm w}}

\def\bmx{\begin{bmatrix}}
\def\emx{\end{bmatrix}}
\def\LM{{\rm\scriptstyle{LM}}}

\def\area{{\rm area}}
\def\LM{{\rm\scriptstyle{LM}}}

\newcommand{\boxs}[1]
{ \multiput(#1)(10,0){2}
 {\line(0,10){10}}
\multiput(#1)(0,10){2}
 {\line(10,0){10}}
}
\begin{document}
\title{Limits of modified higher $q,t$-Catalan numbers}
\author{Kyungyong Lee, Li Li, and Nicholas A. Loehr}
\thanks{Research of K.L. is partially supported by NSF grant DMS 0901367.}
\thanks{This work was partially supported by a grant from the Simons
  Foundation (\#244398 to Nicholas Loehr).}

\address{Department of Mathematics, Wayne State University, Detroit, MI 48202}
\email{{\tt klee@math.wayne.edu}}
\address{Department of Mathematics and Statistics, Oakland University, Rochester, MI 48309}
\email{{\tt li2345@oakland.edu}}
\address{Department of Mathematics, Virginia Tech, Blacksburg, VA 24061}
\email{{\tt nloehr@math.vt.edu}}

\begin{abstract}
The $q,t$-Catalan numbers can be defined using rational functions,
geometry related to Hilbert schemes, symmetric functions, representation
theory, Dyck paths, partition statistics, or Dyck words.
After decades of intensive study, it was
eventually proved that all these definitions are equivalent. In this paper,
we study the similar situation for higher $q,t$-Catalan numbers, where 
the equivalence of the algebraic and combinatorial definitions is still conjectural. We compute
the limits of several versions of the modified higher $q,t$-Catalan numbers and
show that these limits equal the generating function for integer partitions.
We also identify certain coefficients of the
higher $q,t$-Catalan numbers as enumerating suitable integer partitions,
and we make some conjectures on the homological significance of
the Bergeron-Garsia nabla operator.

\end{abstract}
\maketitle

\section{Introduction}
The $q,t$-Catalan numbers and the higher $q,t$-Catalan numbers were
introduced by Garsia and Haiman in the study of symmetric functions
and Macdonald polynomials \cite{GH}. They are polynomials in
$\N[q,t]$ that refine the usual Catalan numbers $\frac{1}{n+1}\binom{2n}{n}$
and higher Catalan numbers $\frac{1}{mn+1}\binom{mn+n}{n}$.  For
a comprehensive introduction to $q,t$-Catalan numbers, the reader
is referred to the book of Haglund \cite{Hbook}. Besides the early
results of Haiman \cite{H98}, Haglund \cite{Ha}, and Garsia and
Haglund \cite{GHag}, there are several recent studies
focused in two directions:

$\bullet$ \emph{Various generalizations.}
Egge, Haglund, Killpatrick, and Kremer studied a generalization
of $q,t$-Catalan numbers obtained by replacing Dyck paths by Schr\"oder
paths~\cite{EHKK}. Loehr and Warrington~\cite{LW0} and Can and Loehr~\cite{CL}
considered the case where Dyck paths are replaced by lattice paths in
a square.  The generalized $q,t$-Fuss-Catalan numbers for finite
reflection groups have been investigated by Stump \cite{S}. Quite recently,
trivariate Catalan numbers defined using trivariate diagonal alternants
have been studied by F. Bergeron and Pr\'eville-Ratelle \cite{BP}.

$\bullet$ \emph{Structural features of the (higher) $q,t$-Catalan numbers.}
N. Bergeron, Descouens, and Zabrocki introduced a filtration of $q,t$-Catalan
numbers connected to the image of $k$-Schur functions under the
nabla operator~\cite{BDZ}. Relations between $q,t$-Catalan numbers and
partition numbers, as well as explicit constructions of the corresponding bases,
have been found by N. Bergeron and Chen~\cite{BC} and Lee and Li~\cite{LL}.
Certain open subvarieties of Hilbert schemes whose affine decompositions are
related to the (higher) $q,t$-Catalan numbers have been constructed by
Buryak~\cite{Bu}.  The significance of $q,t$-Catalan numbers in the study
of the compactified Jacobian of a rational singular curve was revealed
by Gorsky and Mazin \cite{GM}.

A main reason that the (higher) $q,t$-Catalan numbers have so many interesting generalizations and rich structure is because they have several (conjecturally) equivalent definitions that connect different fields of mathematics including combinatorics, symmetric functions, representation
theory, and geometry.  An unsettled conjecture states that definitions of higher $q,t$-Catalan numbers in different fields are all equivalent. Our first main result (Theorem 1.1) shows that all these definitions, after mild modification, have the same limit as $n$ approaches infinity.

Our second main result (Theorem \ref{thm:bound2}) studies the coefficients of the monomial $q^{d_1}t^{d_2}$ in the higher $q,t$-Catalan numbers when the total degree $d_1+d_2$ is close to the maximum possible value $m\binom{n}{2}$. These coefficients are surprisingly simple: they are equal to certain partition numbers. We also give a few conjectures in section 6 including conjectural
minimal generators and  minimal free resolutions for (powers of) diagonal ideals, which may provide a guideline for further exploration.

Before we give the precise statement of our main results, let us review the seven ways of defining the
higher $q,t$-Catalan numbers in (a)--(g) below.
\medskip
\begin{center}
\begin{figure}[h]
\setlength{\unitlength}{0.7pt}
\begin{picture}(300,160)
\put(0,0){\line(2,1){300}}
\put(0,150){\line(0,-1){150}} \put(0,154){$(0,n)$} \put(2,-10){$(0,0)$}
\put(0,50){\line(1,0){100}} \put(100,50){\line(0,+1){30}}
\put(100,80){\line(1,0){30}} \put(130,80){\line(0,+1){20}}
\put(130,100){\line(1,0){40}} \put(170,100){\line(0,+1){50}}
\put(0,150){\line(1,0){300}} \put(300,154){$(mn,n)$}
\put(60,130){\line(0,-1){15}}\put(75,130){\line(0,-1){15}}
\put(60,130){\line(1,0){15}}\put(60,115){\line(1,0){15}}
\put(63,120){$x$}
\put(67,115){\vector(0,-1){65}}\put(67,50){\vector(0,+1){65}}\put(70,83){$l(x)$}
\put(75,123){\vector(1,0){95}}\put(170,123){\vector(-1,0){95}}\put(100,109){$a(x)$}
\put(67,150){\vector(0,-1){20}}\put(67,130){\vector(0,+1){20}}\put(71,136){$l'(x)$}
\put(0,123){\vector(1,0){60}}\put(60,123){\vector(-1,0){60}}\put(15,109){$a'(x)$}
\end{picture}
\label{fig.1}
\caption{Definition of $l(x)$, $a(x)$, $l'(x)$, and $a'(x)$.}
\end{figure}
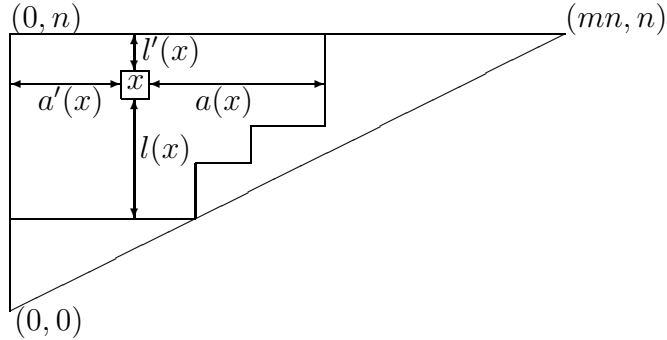
\end{center}
\noindent(a)
Suppose $\lambda$ is an integer partition with Ferrers diagram
$\dg(\lambda)$. Define $\area(\lambda)=|\lambda|=|\dg(\lambda)|$,
the number of cells in the diagram of $\lambda$.
For a cell $x\in\dg(\lambda)$, define the
\emph{leg} $l(x)$, the \emph{arm} $a(x)$,
the \emph{coleg} $l'(x)$, and the \emph{coarm} $a'(x)$
to be the distances shown in Figure~\ref{fig.1}.
Let $\Par_n^{(m)}$ be the set of partitions $\lambda$ such that
$\dg(\lambda)$ fits in the triangle with vertices $(0,0)$, $(0,n)$,
and $(mn,n)$ when drawn as shown in the figure. For such partitions,
define $\area^c(\lambda)=m\binom{n}{2}-\area(\lambda)$ and
$$c_m(\lambda)=|\{x\in\dg(\lambda): m l(x)\le a(x)\le m l(x)+m\}|.$$
For example, when $m=2$ and $\lambda=(7,5,4)$, we have
$|\lambda|=16$ and $c_2(\lambda)=13$.

Define the \emph{partition version of the higher $q,t$-Catalan numbers} by
\[ PC^{(m)}_n(q,t)=\sum_{\lambda\in\Par_n^{(m)}}
q^{\area^c(\lambda)}t^{c_m(\lambda)}. \]
For example, $PC^{(3)}_2(q,t)=q^3+q^2t+qt^2+t^3$ and
$$PC^{(2)}_3(q,t)=q^6+q^5t+q^4t^2+q^4t+q^3t^3+q^3t^2+q^2t^2+q^2t^3+qt^4+q^2t^4
+qt^5+t^6.$$

\noindent(b) An \emph{$m$-Dyck word} is a sequence
$\gamma=(\gamma_0,\gamma_1, \ldots,\gamma_{n-1})$
such that $\gamma_i\in\N=\{0,1,2,\ldots\}$,
$\gamma_0=0$, and $\gamma_{i+1}\leq \gamma_i+m$ for $0\leq i<n-1$.
Let $\Gamma_n^{(m)}$ be the set of $m$-Dyck words of length $n$.
For such a word $\gamma$, define $\area(\gamma)=\sum_{i=0}^{n-1} \gamma_i$.
As in~\cite{L}, define
${\rm dinv_m}(\gamma)=\sum_{0\le i<j<n} sc_m(\gamma_i-\gamma_j)$,
where $$
sc_m(p)=
\left\{
 \begin{array}{ll}
   m+1-p, &\hbox{if $1\le p\le m$};\\
   m+p,  &\hbox{if $-m\le p\le 0$};\\
   0, &\hbox{for all other $p$}.
 \end{array}
 \right.
$$
For example, $\gamma=(0,2,0,1,1)\in\Gamma_5^{(2)}$ has
$\area(\gamma)=4$ and $\dinv_2(\gamma)=13$.

Define the \emph{word version of the higher $q,t$-Catalan numbers} by
\[ WC^{(m)}_n(q,t)=\sum_{\gamma\in\Gamma_n^{(m)}}
  q^{\area(\gamma)}t^{\dinv_m(\gamma)}. \]

\noindent(c) An \emph{$m$-Dyck path of order $n$}
is a lattice path $\pi$ from $(0,0)$ to
$(mn,n)$ using north and east steps such that the path
never goes below the diagonal line segment with endpoints $(0,0)$ and
$(mn,n)$.  Let $\mathcal{D}^{(m)}_n$ be the set of such $m$-Dyck paths.
For such a path $\pi$, let $\area(\pi)$ be the number of complete
unit squares between $\pi$ and the diagonal. Define the \emph{$m$-bounce
statistic} $b_m(\pi)$ as follows.  Set $v_i = 0$ for all negative integers $i$.
Starting from $(0, 0)$, construct a \emph{bounce path} by induction
on $i\ge 0$. In the $(i+1)$th step, move north from the current
position $(u,v)$ until hitting an east step of the $m$-Dyck path
that starts on the line $x=u$, and define the distance traveled to be $v_i$.
Then move east from this position $v_i + v_{i-1} + \cdots + v_{i-m+1}$ units.
Continue bouncing until reaching $(mn, n)$. (In
fact, it suffices to stop once we reach the horizontal line $y = n$.)
Then $b_m(\pi)=\sum_{k\ge 0}kv_k.$  For example, the path
$\pi\in\mathcal{D}^{(2)}_5$ in Figure~\ref{fig:bounce} has $\area(\pi)=4$,
$(v_0,v_1,\ldots,v_5)=(2,0,1,1,1,0)$, and $b_2(\pi)=9$.

Define the \emph{Dyck path version of the higher $q,t$-Catalan numbers} by
\[ DC^{(m)}_n(q,t)=\sum_{\pi\in \mathcal{D}^{(m)}_n}
 q^{b_m(\pi)}t^{\area(\pi)}. \]
\begin{center}
\begin{figure}
 \setlength{\unitlength}{0.5pt}
\begin{picture}(300,150)
\put(0,0){\line(2,1){300}}
\put(0,150){\line(0,-1){150}}
\put(0,150){\line(1,0){300}}
{\linethickness{1.5pt}
\put(0,0){\line(0,1){61}}
\put(0,60){\line(1,0){121}} \put(120,60){\line(0,1){31}}
\put(120,90){\line(1,0){31}} \put(150,90){\line(0,1){31}}
\put(150,120){\line(1,0){61}} \put(210,120){\line(0,1){31}}
\put(210,150){\line(1,0){90}}
} %
\put(0,30){\line(1,0){60}}
\put(0,90){\line(1,0){180}}
\put(0,120){\line(1,0){240}}
\put(30,150){\line(0,-1){120}}
\put(60,150){\line(0,-1){120}}
\put(90,150){\line(0,-1){90}}
\put(120,150){\line(0,-1){90}}
\put(150,150){\line(0,-1){60}}
\put(180,150){\line(0,-1){60}}
\put(240,150){\line(0,-1){30}} %
\put(-60,-10){$(0,0)$}
\put(310,140){$(mn,n)$}
\end{picture}
\caption{An $m$-Dyck path.}
\label{fig:bounce}
\end{figure}
\end{center}

\noindent(d) The $q,t$-Catalan numbers may be defined using symmetric
functions, as follows. This discussion assumes the reader is familiar with the
elementary symmetric functions $e_n$, the modified Macdonald polynomials
$\tilde{H}_{\mu}$, and the Hall scalar product $\langle\cdot,\cdot\rangle$
on symmetric functions; see~\cite{Hbook} or~\cite[\S3.5.5]{H03} for details.
For any integer partition $\mu$, define $n(\mu)=\sum_{x\in\dg(\mu)} l(x)$
and $n(\mu')=\sum_{x\in\dg(\mu')} l(x)=\sum_{x\in\dg(\mu)} a(x)$, where
$\mu'$ denotes the transpose of $\mu$. Define $T_{\mu}=q^{n(\mu')}t^{n(\mu)}$.
Let $\Lambda$ denote the ring of symmetric functions with coefficients
in the field $F=\Q(q,t)$.  The Bergeron-Garsia \emph{nabla operator}~\cite{BG}
is the unique $F$-linear map $\nabla$ on $\Lambda$ that acts on the modified
Macdonald basis via $\nabla(\tilde{H}_{\mu})=T_{\mu}\tilde{H}_{\mu}$
for all partitions $\mu$. For $m\in\N^+=\{1,2,3,\ldots\}$,
$\nabla^m$ denotes the composition
of $m$ copies of the operator $\nabla$. We now define the
\emph{symmetric function version of the higher $q,t$-Catalan numbers} by
\[ SC_n^{(m)}(q,t)=\langle \nabla^m(e_n),e_n\rangle. \]

\noindent(e) The higher $q,t$-Catalan numbers were originally defined
by Garsia and Haiman in~\cite{GH} as sums of rational functions in $\Q(q,t)$
constructed from integer partitions.  Recall that $\mu\vdash n$ means that
$\mu$ is an integer partition of $n$. With $T_{\mu}$ defined as in (d),
we further define
$$
\aligned
& B_\mu=\sum_{x\in\dg(\mu)}q^{a'(x)}t^{l'(x)},
& \quad \Pi_\mu=\prod_{x\in\dg(\mu)\setminus\{(0,0)\}}(1-q^{a'(x)}t^{l'(x)}),
\\
& w_\mu=\prod_{x\in\dg(\mu)}[(q^{a(x)}-t^{l(x)+1})(t^{l(x)}-q^{a(x)+1})].\\
\endaligned
$$
Then the \emph{rational function version of the higher $q,t$-Catalan numbers}
is defined by
$$RC^{(m)}_n(q,t)=\sum_{\mu\vdash n} (1-q)(1-t) T_\mu^{m+1}B_\mu\Pi_\mu/w_\mu.$$

\noindent(f) For fixed $n\in\N^+$, consider the polynomial ring
$\C[\mathbf{x},\mathbf{y}]=\mathbb{C}[x_1,y_1,\cdots,x_n,y_n]$.
$S_n$ acts diagonally on this ring by the rule
$w\cdot x_i=x_{w(i)}$, $w\cdot y_i=y_{w(i)}$ for $w\in S_n$
and $1\leq i\leq n$. A polynomial $f\in\C[\mathbf{x},\mathbf{y}]$
is called \emph{alternating} iff $w\cdot f=\sgn(w)f$ for all $w\in S_n$.
Let $I$ be the ideal in $\C[\mathbf{x},\mathbf{y}]$
generated by all alternating polynomials, and let
$\mathfrak{m}$ be the maximal ideal generated by $x_1,y_1,\ldots,x_n,y_n$.
We write $I=I_n$ and $\mathfrak{m}=\mathfrak{m}_n$ if it is
necessary to indicate the number of variables.
Let $M^{(m)}=I^m/\mathfrak{m}I^m$ for $m\in\N$, and for simplicity, let $M=M^{(1)}$.
 Given a monomial $f=x_1^{a_1}y_1^{b_1}\cdots
x_n^{a_n}y_n^{b_n}\in\C[\textbf{x},\textbf{y}]$, we define the
\emph{bidegree of $f$} to be the ordered pair
$(\sum_{i=1}^n a_i,\sum_{i=1}^n b_i)$.
We say that a polynomial in $\C[\textbf{x},\textbf{y}]$ is
\emph{bihomogeneous of bidegree $(d_1,d_2)$} if all its monomials
have the same bidegree $(d_1,d_2)$. Then
$I^m$ and $M^{(m)}$ become doubly-graded $S_n$-modules
by taking bidegrees in the $x$-variables and the $y$-variables.
Let $M^{(m)}_{u,v}$ denote the bihomogeneous component of $M^{(m)}$
of bidegree $(u,v)$.
Define the \emph{algebraic version of the higher $q,t$-Catalan
numbers} by
$$AC_n^{(m)}(q,t)=\sum_{u\geq 0}\sum_{v\geq 0} q^u t^v \dim M_{u,v}^{(m)}.$$
For more information, see \cite[Section 3]{GH}.

\noindent(g) Finally we state the geometric definition
(see~\cite{H01,H02} for more details).
Let $Z_n$ be the zero fiber of the Hilbert-Chow morphism ${\rm
Hilb}^n(\C^2)\to{\rm Sym}^n(\C^2)$, let
$\mathcal{O}(1)$ be the restriction of the ample line bundle on
${\rm Hilb}^n(\C^2)$ induced by the isomorphism ${\rm
Hilb}^n(\C^2)\cong {\rm Proj}(T)$, where $T=\bigoplus_{d\ge 0} A^d$
and $A=\C[x_1,y_1,\dots,x_n,y_n]^\varepsilon$ is the
space of $S_n$-alternating elements. For any $m\in\N^+$, let $\mathcal{O}(m)=\mathcal{O}(1)^{\otimes m}$.
The set of global sections $H^0(Z_n,\mathcal{O}(m))$ is a bigraded vector
space. Define the \emph{geometric version of the higher $q,t$-Catalan numbers}
by $$GC^{(m)}_n(q,t)=\sum_{u,v}q^ut^v\dim H^0(Z_n,\mathcal{O}(m))_{u,v}.$$

It is conjectured that the seven definitions (a)--(g) of higher $q,t$-Catalan
numbers are all equivalent. This conjecture is supported by explicit
computations for small values of $m$ and $n$. It has been proved
that for all $m,n\in\N^+$,
\begin{equation}\label{eq:comb-eqs}
 PC_n^{(m)}(q,t)=WC_n^{(m)}(q,t)=DC_n^{(m)}(q,t) \mbox{ and }
\end{equation}
\begin{equation}\label{eq:alg-eqs}
 SC_n^{(m)}(q,t)=RC_n^{(m)}(q,t)=AC_n^{(m)}(q,t)=GC_n^{(m)}(q,t).
\end{equation}
We discuss the proofs of these equalities in the appendix (\S7).
It remains to be proved that the three combinatorial definitions agree
with the four algebraic and geometric definitions. This conjecture has
already been proved for certain specializations of the parameters $q$ and $t$.
For instance, using \cite[Theorem 4.4]{GH} and definitions (a) and (c)
above, we find that
\[ RC_n^{(m)}(q,1)=RC_n^{(m)}(1,q)
  =\sum_{\pi\in\mathcal{D}^{(m)}_n} q^{\area(\pi)}
  =PC_n^{(m)}(q,1)=DC_n^{(m)}(1,q). \]
Upon setting $t=1/q$, we see from~\cite[Corollary 4.1]{GH}
and \cite[\S3.3]{L} that
$$RC^{(m)}_n(q,1/q)q^{m\binom{n}{2}}
 =WC^{(m)}_n(q,1/q)q^{m\binom{n}{2}}
 =DC^{(m)}_n(q,1/q)q^{m\binom{n}{2}}=
\frac{1}{[mn+1]_q}\dqbin{mn+n}{n}{q}.$$

This paper studies the limiting behavior, as $n$ tends to infinity, of the ``modified''
higher $q,t$-Catalan numbers given by
$$q^{m\binom{n}{2}}PC_n^{(m)}(q^{-1},t), \quad
q^{m\binom{n}{2}}DC_n^{(m)}(t,q^{-1}),\mbox{ and }
q^{m\binom{n}{2}}AC_n^{(m)}(q^{-1},t).$$
We will show that all of these polynomials have as their limit
the famous generating function $\prod_{i=1}^{\infty} (1-tq^i)^{-1}$,
which enumerates integer partitions by area and number of parts.
(Here we are taking limits in a formal power series ring,
which means that for each fixed monomial $q^at^b$, the coefficient
of this monomial becomes stable for sufficiently large $n$.)
The following is our first main theorem, which is the combination of Proposition \ref{mainthm:part1}, Proposition \ref{mainthm:part2} and Corollary \ref{02102011vv}.
\begin{thm}\label{mainthm}
For any positive integer $m$, we have
$$\aligned
&\lim_{n\to\infty} q^{m\binom{n}{2}}PC_n^{(m)}(q^{-1},t)
=\lim_{n\to\infty} q^{m\binom{n}{2}}DC_n^{(m)}(q^{-1},t)
=\lim_{n\to\infty} q^{m\binom{n}{2}}AC_n^{(m)}(q^{-1},t)\\
&=\prod_{i=1}^\infty (1-tq^i)^{-1}
  =\sum_{\mu\in\Par} q^{\area(\mu)}t^{\ell(\mu)},
\endaligned
$$
where $\Par$ is the set of all integer partitions,
and $\ell(\mu)$ is the number of parts of $\mu$.
\end{thm}
Thus we have the following corollary using (\ref{eq:comb-eqs}) and (\ref{eq:alg-eqs}).
\begin{cor}
For any positive integer $m$, we have
$$\aligned &\phantom{=\;}\lim_{n\to\infty} q^{m\binom{n}{2}}PC_n^{(m)}(q^{-1},t)
=\lim_{n\to\infty} q^{m\binom{n}{2}}WC_n^{(m)}(q^{-1},t)
=\lim_{n\to\infty} q^{m\binom{n}{2}}GC_n^{(m)}(q^{-1},t)\\
&=\lim_{n\to\infty} q^{m\binom{n}{2}}SC_n^{(m)}(q^{-1},t)
=\lim_{n\to\infty} q^{m\binom{n}{2}}RC_n^{(m)}(q^{-1},t)
=\lim_{n\to\infty} q^{m\binom{n}{2}}AC_n^{(m)}(q^{-1},t)\\
&=\lim_{n\to\infty} q^{m\binom{n}{2}}DC_n^{(m)}(t,q^{-1})
=\prod_{i=1}^\infty (1-tq^i)^{-1}
  =\sum_{\mu\in\Par} q^{\area(\mu)}t^{\ell(\mu)}.
 \endaligned$$
\end{cor}

The result for $AC^{(1)}_n$ can also be obtained from a result of
N. Bergeron and Chen \cite[Corollary 8.3]{BC}.

Our second main theorem identifies the dimensions of $M^{(m)}_{d_1,d_2}$, which are the coefficients of certain
terms $q^{d_1}t^{d_2}$ in $AC^{(m)}_n(q,t)$, as partition numbers. The partition number  $p(\delta,k)$ is the number of partitions of $k$ into at most $\delta$ parts.
By convention, $p(0,k)=0$ for $k>0$, and $p(\delta,0)=1$
for $\delta\ge 0$.

\begin{thm}\label{thm:bound2} Let $n\ge6$ and $m$ be positive integers,
and let $k, d_1, d_2$ be nonnegative integers such that
$k=m\binom{n}{2}-d_1-d_2\le n-6$. Let $\delta=\min(d_1,d_2)$. Then
$$\dim M_{d_1,d_2}^{(m)}=p(\delta,k).$$
\end{thm}

The paper is organized as follows. 
In \S2 we prove the combinatorial part of our main theorem.
In \S3 we introduce further notation, background, and preliminary results.
In \S4 we prove the algebraic part of the main theorem.
In \S5, we extend the method used in \S4 to prove Theorem \ref{thm:bound2}.  In \S6 we give some related conjectures.
In \S7 we indicate the proofs of the equalities stated
in~\eqref{eq:comb-eqs} and~\eqref{eq:alg-eqs}.

\noindent \emph{Acknowledgements.}  We are grateful to Drew Armstrong, Nantel Bergeron, Fran\c{c}ois Bergeron, Alex Woo, and Alex Yong for helpful suggestions and correspondence.
The authors are also grateful to the
anonymous referee for many useful comments.

\section{Limits of the Modified Combinatorial Higher $q,t$-Catalan Numbers}
In this section, we study the limiting behavior of the modified $PC^{(m)}_n$ and $DC^{(m)}_n$. Even though logically it suffices to study one of them because they are equal (see \eqref{eq:comb-eqs}), we feel that both proofs have their own interest to be presented here.  We first recall the following theorem~\cite[Thm. 3]{LW}.
\begin{thm}\label{thm:LW}
For $\lambda\in\Par$ and $m\in\R^+$, define $h^+_m(\lambda)$ to
be the number of cells $x\in\dg(\lambda)$ such that
$\frac{a(x)}{l(x)+1}\le m<\frac{a(x)+1}{l(x)}$. Then
$$
\sum_{\lambda\in{\rm Par}} q^{\area(\lambda)}t^{h^+_m(\lambda)}=\prod_{i=1}^\infty\frac{1}{1-tq^i}=\sum_{\mu\in{\rm Par}} q^{\area(\mu)}t^{\ell(\mu)}.
$$
\end{thm}

\begin{prop}\label{mainthm:part1} {\rm (i)} For all $m,n\in\N^+$,
$$q^{m\binom{n}{2}}PC_n^{(m)}(q^{-1},t)
=\sum_{\lambda\in\Par_n^{(m)}}q^{\area(\lambda)}t^{c_m(\lambda)}.$$

{\rm (ii)} For all $m\in\N^+$,
$$\lim_{n\to\infty} q^{m\binom{n}{2}}PC_n^{(m)}(q^{-1},t)
=\sum_{\lambda\in\Par}
q^{\area(\lambda)}t^{c_m(\lambda)}=\prod_{i=1}^\infty \frac{1}{1-tq^i}.$$
\end{prop}
\begin{proof} (i) is straightforward. For (ii),
if we increase $n$ by $1$, a partition $\lambda$ in the $mn\times n$ triangle
will also fit into the $m(n+1)\times (n+1)$ triangle, and the two statistics
$\area(\lambda)$ and $c_m(\lambda)$ do not change with $n$.
Since all integer partitions of a fixed area will fit in the triangle
for sufficiently large $n$,
the first equality follows from (i).
The second equality follows from Theorem~\ref{thm:LW}
and the observation that $h_m^+(\lambda)=c_m(\lambda)$.
\end{proof}

\begin{prop}\label{mainthm:part2} {\rm (i)}
For all $m,n\in\N^+$,
$$q^{m\binom{n}{2}}DC_n^{(m)}(t,q^{-1})
=\sum_{\pi\in\mathcal{D}^{(m)}_n}q^{\area^c(\pi)}t^{b_m(\pi)},$$
where $\area^c(\pi)$ is the number of lattice squares in the
$mn\times n$ triangle above $\pi$.

{\rm (ii)} For all $m\in\N^+$,
$$\lim_{n\to\infty} q^{m\binom{n}{2}}DC_n^{(m)}(t,q^{-1})
=\sum_{\lambda\in\Par}q^{\area(\lambda)}t^{\ell(\lambda)}
=\prod_{i=1}^\infty \frac{1}{1-tq^i}.$$
\end{prop}
\begin{proof} (i) is straightforward.
For (ii), note that each $m$-Dyck path $\pi\in\mathcal{D}_n^{(m)}$
determines an integer partition $\lambda=\lambda(\pi)$ whose diagram
consists of the squares above $\pi$ in the $mn\times n$ triangle.
For each $n$, there is an injection $\mathcal{D}_n^{(m)}
\rightarrow \mathcal{D}_{n+1}^{(m)}$ that adds one north step to the beginning
of $\pi$ and adds $m$ east steps to the end of $\pi$. This injection
preserves both $\lambda(\pi)$ and $\area^c(\pi)=\area(\lambda(\pi))$,
but the value of the bounce statistic $b_m(\pi)$ may change.
However, as $n$ continues to increase, the
bounce statistic will eventually stabilize. More specifically, once
$n\geq 2\area^c(\pi)$, it is routine to check that the bounce path
will satisfy $v_0=n-\ell(\lambda(\pi))$, $v_1=\ell(\lambda(\pi))$,
and $v_i=0$ for all $i>2$; the key observation is that the first
horizontal move (of length $v_0$) moves to the right of all the
squares in $\dg(\lambda(\pi))$.  It follows that $b_m(\pi)=\ell(\lambda(\pi))$
for such $n$. By fixing the area of the partition outside $\pi$ and
taking $n$ larger than twice this area, we see as in the previous
proposition that the indicated limit holds.
\end{proof}

\section{Notation and background for $AC^{(m)}_n$}
\subsection{Notation}
\begin{itemize}
\item For $k,b\in \N^+$, denote the set of integer partitions of $k$
by $\Par(k)$, and denote the set of integer partitions of $k$ into at most
$b$ parts by $\Par(b,k)$. More explicitly,
$\Par(k) =\{\nu=(\nu_1,\nu_2,\dots,\nu_\ell)|\, \nu_i\in\N^+,\,
\nu_1\le\nu_2\le\cdots\le\nu_\ell,  \,\nu_1+\nu_2+\cdots+\nu_\ell=k\}$ and
$\Par(b,k) =\{\nu=(\nu_1,\nu_2,\dots,\nu_\ell)\in\Par(k)| \,\ell\le b\}$.
By convention, $\Par(0)=\{0\}$, $\Par(0,k)=\emptyset$ for $k>0$,
and $\Par(h,0)=\{0\}$ for all
$h\ge 0$ (where $\{0\}$ is a set with one element).
Let $p(k)$ and $p(b,k)$ be the cardinalities of $\Par(k)$ and $\Par(b,k)$,
respectively. In other words, $p(k)$ is the number of partitions of $k$
and $p(h,k)$ is the number of partitions of $k$ into at most $h$ parts.
By the above conventions, $p(0)=1$, $p(0,k)=0$ for $k>0$, and $p(h,0)=1$
for all $h\ge 0$.

\item Let $\C[\mathbf{\rho}]=\C[\rho_1,\rho_2,\dots]$ be the polynomial
ring with countably many variables $\rho_i$, for $i\in\N^+$.
As a convention, we set $\rho_0=1$. For a partition
$\nu=(\nu_1,\nu_2,\dots,\nu_\ell)\in\Par(k)$, define
$\rho_\nu=\rho_{\nu_1}\rho_{\nu_2}\cdots\rho_{\nu_\ell}\in \C[\mathbf{\rho}]$.
Define the \emph{weight} of a monomial $c\rho_{i_1}\cdots\rho_{i_\ell}$
(where $c\in\C\setminus\{0\}$) to be
$i_1+\dots+i_\ell$. For $w\in\N$, define $\C[\mathbf{\rho}]_w$ to be
the subspace of $\C[\mathbf{\rho}]$ spanned by monomials of weight $w$. For $f\in \C[\rho]$, there is a unique expression $f=\sum_{w=0}^\infty \{f\}_w$ with $\{f\}_w\in \C[\rho]_w$, and we call $\{f\}_w$ the \emph{weight-w} part of $f$.
\item For $P=(a,b)\in\N\times\N$,
we write $|P|=a+b$, $|P|_x=a$, and $|P|_y=b$.
\item For $n\in \N^+$, define
$$\D_n=\{D\subset\N\times\N: |D|=n\}.  $$
For $D\in\D_n$, we write $D=\{P_1,P_2,\ldots,P_n\}$ where each
$P_i=(a_i,b_i)\in\N\times\N$. Unless otherwise specified, we always choose
notation so that $P_1,\ldots,P_n$ are in increasing graded lexicographic order.
This means that $P_1<P_2<\cdots<P_n$, where
$$ \textrm{ $(a,b)<(a',b')$
if $a+b<a'+b'$, or if $a+b=a'+b'$ and $a<a'$.}$$
To visualize a set $D\in\D_n$, we can draw a square grid on which
we plot the $n$ ordered pairs in $D$.  For example, in the
following picture, the horizontal and vertical bold lines represent
the $x$-axis and $y$-axis,  and $D=\big{\{}
(0,0), (1,0), (1,1), (2,0), (3,0)\big{\}}$.

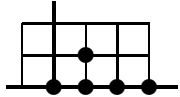
    \setlength{\unitlength}{1.2pt}
    \begin{picture}(100,30)(-100,0)
    \put(10,0){\circle*{5}}\put(20,0){\circle*{5}}
    \put(20,10){\circle*{5}}\put(30,0){\circle*{5}}\put(40,0){\circle*{5}}
    \boxs{0,0}\boxs{10,0}\boxs{20,0}\boxs{30,0}
    \boxs{0,10}\boxs{10,10}\boxs{20,10}\boxs{30,10}
    \linethickness{1pt}\put(10,0){\line(0,1){27}}
    \linethickness{1pt}\put(-5,0){\line(1,0){54}}

    \end{picture}

\item Given $D=\{P_1,\dots,P_n\}\in\D_n$, define the \emph{total degree}, \emph{$x$-degree},
\emph{$y$-degree}, and \emph{bidegree} of $D$ to be $\sum_{i=1}^n \left(|P_i|_x + |P_i|_y\right)$,
$\sum_{i=1}^n |P_i|_x$, $\sum_{i=1}^n |P_i|_y$, and
$(\sum_{i=1}^n |P_i|_x,\sum_{i=1}^n |P_i|_y)$, respectively. Then the $x$-degree (resp. $y$-degree) of $D$ will be denoted by $d_1(D)$ (resp. $d_2(D)$). Let $k(D)=\binom{n}{2}-d_1(D)-d_2(D)$.

\item The diagonal ideal $I$ of $\C[\mathbf{x},\mathbf{y}]$
and the bigraded $\C$-vector space
$M=\bigoplus_{d_1,d_2\in\N} M_{d_1,d_2}$ were defined in \S1(f).
The ideal generated by all homogeneous elements in $I$ of total degree less
than $d$ is denoted by $I_{<d}$.

\item For $D=\{(a_1,b_1),...,(a_n,b_n)\}\in\D_n$,
the alternating polynomial $\Delta(D)\in\C[\textbf{x},\textbf{y}]$
is defined by
$$\Delta(D)=\det[x_i^{a_j}y_i^{b_j}]_{1\leq i,j\leq n}=
   \det\begin{vmatrix}
     x_1^{a_1}y_1^{b_1}&x_1^{a_2}y_1^{b_2}&...&x_1^{a_n}y_1^{b_n}\\
      \vdots&\vdots &\ddots &\vdots\\
    x_n^{a_1}y_n^{b_1}&x_n^{a_2}y_n^{b_2}&...&x_n^{a_n}y_n^{b_n}\\
   \end{vmatrix}.
$$
Note that $\Delta(D)$ is bihomogeneous of bidegree equal to
the bidegree of $D$.

\item Given two polynomials $f,g\in I^m$ of the same bidegree $(d_1,d_2)$, let $\bar{f},\bar{g}$
be the corresponding elements
in $M^{(m)}_{d_1,d_2}$. For $m=1$, we say that $$f\equiv g \quad\textrm{(modulo lower
degrees)}$$ if $\bar{f}=\bar{g}$ in $M_{d_1,d_2}$, or, equivalently, if $f-g$ is in $I_{<d_1+d_2}$.

\item Given $d_1+d_2=\binom{n}{2}$, take arbitrary $D=\{P_1,\dots,P_n\}\in\D_n$
of bidegree $(d_1,d_2)$ such that $|P_i|=i-1$. Define $f_{d_1,d_2}$ to be
the equivalence class of $\Delta(D)$ in $M_{d_1,d_2}$.
By \cite[Lemma 16]{LL}, this equivalence class is independent of the choice
of $D$.

\end{itemize}

\subsection{Properties of the module $M_{d_1,d_2}$} This subsection is organized as follows. First, to be self-contained, we review the definitions of staircase forms, block diagonal forms, and partition types introduced in \cite{LL}. The reader is suggested to look at Example \ref{eg:staircase form} while reading these definitions.
Then we recall the map $\bar{\varphi}$ defined in \cite{LL} and use its injectivity to prove Lemma \ref{lem:injective} that we shall use later.

\begin{defnprop}[{\cite[Definition-Proposition 6]{LL}}]\label{definitionofstaircaseform}
Let $D=\{P_1,\dots,P_n\}\in\mathfrak{D}_n$, and write $P_i=(a_i,b_i)$.
Then there is an $n\times n$ matrix $S$ whose $(i,j)$-entry is
$$\left\{
           \begin{array}{ll}
0, &\hbox{if $i\le |P_j|$};\\
z_{i1}z_{i2}\cdots z_{i,|P_j|}  \hbox{ where $z_{i\ell}$ is either $x_i-x_\ell$ or $y_i-y_\ell$},&\hbox{otherwise,}
           \end{array}
         \right.$$
for all $1\le i,j\le n$, such that $\det(S)\equiv\Delta(D)$
(modulo lower degrees).  We call $S$ a \emph{staircase form} of $D$.
\end{defnprop}

\begin{defn}\label{df:block diagonal} Let $D$ and $S$ be defined as
in Definition-Proposition \ref{definitionofstaircaseform}.
Consider the set $\{j: |P_j|=j-1\}=\{r_1< r_2<\dots< r_\ell\}$ and
define $r_{\ell+1}=n+1$. For  $1\le t\le \ell$, define the $t$-th
block $B_t$ of $S$ to be the square submatrix of $S$ of size $(r_{t+1}-r_t)$
whose upper-left corner is the ($r_t,r_t$)-entry.  Define the \emph{block
diagonal form} $B(S)$ of $S$ to be the block diagonal matrix
$\hbox{diag}(B_1,\dots,B_\ell)$.
\end{defn}

\begin{defn}\label{minimalst}
Let $S$ be a staircase form, $B(S)$ be its block diagonal form with blocks $B_1,\dots,B_\ell$. For $1\le t\le \ell$, let $\mu_t$ be the number of nonzero entries in block $B_t$ that are strictly above the diagonal, i.e., the number of nonzero $i,j$-entries in $B_t$ where $j>i$. Eliminating zeros in $(\mu_1,\dots,\mu_\ell)$ and then rearranging the sequence in ascending order, we obtain a partition of $k$, denoted by $\mu(S)$. We say that $S$ is of partition type $\mu(S)$. We call a block $B_t$ \emph{minimal} if every $(i,j)$-entry ($j>i+1$) that lies in $B_t$ is zero. We call $S$ a minimal staircase form if all the blocks in $B(S)$ are minimal.
We say $D\in\mathfrak{D}_n$ is of partition type $\mu(S)$ if $S$ is a staircase form of $D$. (Note that the partition type does not depend on the choice of $S$.)
\end{defn}

\begin{exmp}\label{eg:staircase form}
(i) Let $D=\{(0,0),(0,1),(0,2),(1,1)\}\in \mathfrak{D}_4$. We list here $\Delta(D)$ and a possible staircase form $S$ together with the corresponding block diagonal forms $B(S)$. In this example, $D$ is of partition type $(1)$, $S$ is a minimal staircase form, and $B(S)$ has two blocks of size 1 and one block of size 2.
{\small
$$\Delta(D)=\begin{vmatrix}
1 & y^{}_{1} &y_{1}^2 &x^{}_{1} y^{}_{1} \\
1 & y^{}_{2} &y_{2}^2 &x^{}_{2} y^{}_{2} \\
1 & y^{}_{3} &y_{3}^2 &x^{}_{3} y^{}_{3}\\
1 & y^{}_{4} &y_{4}^2 &x^{}_{4} y^{}_{4}\\
\end{vmatrix},
\;
S=\begin{bmatrix}
1 & 0      &0            &0   \\
1 & y_{21} &0            &0   \\
1 & y_{31} &y_{31}y_{32} &x_{31}y_{32} \\
1 & y_{41}  &y_{41}y_{42} &x_{41}y_{42} \\
\end{bmatrix},
\;
B(S)=\begin{bmatrix}
1 & 0      &0            &0   \\
0 & y_{21} &0            &0   \\
0 & 0 &y_{31}y_{32} &x_{31}y_{32} \\
0 & 0  &y_{41}y_{42} &x_{41}y_{42} \\
\end{bmatrix},
$$
}
where $x_{ij}=x_i-x_j$ and $y_{ij}=y_i-y_j$.

\noindent(ii) Let $D=\{(0,0),(0,1),(1,0),(1,2),(2,1),(3,0)\}\in \mathfrak{D}_6$. A staircase form $S$ and the corresponding block diagonal form $B(S)$ are given below. Then $D$ is of partition type $(1,3)$, $B(S)$ has three blocks of sizes $1, 2, 3$ respectively, and $S$ is not a minimal staircase form (because the ($4,6$)-entry of $B(S)$ is $x_{41}x_{42}x_{43}\neq 0$, therefore the block $B_3$ is not minimal).
{\tiny
$$
 S=\begin{bmatrix}
1 & 0      &0           &0 &0 &0\\
1 & y_{21}&x_{21} &0 &0 &0 \\
1 & y_{31}&x_{31} &0 &0 & 0\\
1 & y_{41}&x_{41} &y_{41}y_{42}x_{43} &y_{41}x_{42}x_{43} &x_{41}x_{42}x_{43}\\
1 & y_{51}&x_{51} &y_{51}y_{52}x_{53} &y_{51}x_{52}x_{53} &x_{51}x_{52}x_{53}\\
1 & y_{61}&x_{61} &y_{61}y_{62}x_{63} &y_{61}x_{62}x_{63} &x_{61}x_{62}x_{63}\\
\end{bmatrix},
\;
B(S)=\begin{bmatrix}
1 & 0      &0           &0 &0 &0\\
0 & y_{21}&x_{21} &0 &0 &0 \\
0 & y_{31}&x_{31} &0 &0 & 0\\
0 & 0        &0         &y_{41}y_{42}x_{43} &y_{41}x_{42}x_{43} &x_{41}x_{42}x_{43}\\
0 & 0        &0         &y_{51}y_{52}x_{53} &y_{51}x_{52}x_{53} &x_{51}x_{52}x_{53}\\
0 & 0        &0         &y_{61}y_{62}x_{63} &y_{61}x_{62}x_{63} &x_{61}x_{62}x_{63}\\
\end{bmatrix}.
$$
}

\end{exmp}

\begin{thm}[{\cite[Theorem 5]{LL}}]\label{thm:LL}
Let $n$ be a positive integer, and let $d_1,d_2, k$ be non-negative integers
such that $k=\binom{n}{2}-d_1-d_2$.
Define $\delta=\min(d_1,d_2)$. Then $\dim M_{d_1, d_2}\le p(\delta,k)$,
and equality holds here if and only if
either $k\leq n-3$, or $k=n-2$ and $\delta=1$, or $\delta=0$.
\end{thm}

Recall some definitions in \cite{LL}.  For $b\in\N$ and $\w\in\Z$, define
$$h(b,\w)=\big{\{}(1+\rho_1+\rho_2+\cdots)^b\big{\}}_\w.$$
For $D=\{P_1,\ldots,P_n\}\in\D_n$, define $\varphi(D)$ to be
$$(-1)^{k(D)}\det\begin{vmatrix}
h(b_1, -|P_1|)& h(b_1,1-|P_1|)& h(b_1,2-|P_1|)&\cdots&h(b_1,n-1-|P_1|)\\
h(b_2,-|P_2|)& h(b_2,1-|P_2|)& h(b_2,2-|P_2|)&\cdots&h(b_2,n-1-|P_2|)\\
\vdots&\vdots&\vdots&\ddots&\vdots\\
h(b_n,-|P_n|)& h(b_n,1-|P_n|)& h(b_n,2-|P_n|)&\cdots&h(b_n,n-1-|P_n|)\\
\end{vmatrix}.$$
It is proved in~\cite[Lemma 47]{LL} that $\varphi$ induces a well-defined linear map $\bar{\varphi}:M_{d_1,d_2}\rightarrow \C[\rho]_{\binom{n}{2}-d_1-d_2}$.
We have conjectured that $\bar{\varphi}$ is injective and proved the
injectivity under the condition $\binom{n}{2}-d_1-d_2\le n-3$ and
$d_2\le d_1$ \cite[Conjecture 48, Theorem 43, Theorem 44]{LL}.
With a slight modification, we can prove the injectivity under the sole condition $\binom{n}{2}-d_1-d_2\le n-3$ without the constraint $d_2\le d_1$. (We briefly explain the modification using the terminology in \cite{LL}: assume now $d_2>d_1$. It suffices to prove that, for each $\nu\in\Pi_{d_1,k}$, there exists an alternating polynomial $g_\nu$ such that the leading monomial ${\LM}(\varphi(g_\nu))=\rho_\nu$. In fact, such a $g_\nu$ can be obtained, up to a sign, by switching $x$- and $y$- coordinates of the $f_\nu$ constructed for $M_{d_2,d_1}$ in \cite[Theorem 44]{LL}.)
\begin{lem}\label{lem:injective}
Suppose $0\le \binom{n-1}{2}-d'_1-d'_2\le n-4$, $d'_1\le d_1$, $d'_2\le d_2$, and $d'_1+d'_2+(n-1)=d_1+d_2$.
Let $M'_{d'_1,d'_2}$ and $M_{d_1,d_2}$ be the indicated bigraded
components of $I_{n-1}/\mathfrak{m}_{n-1}I_{n-1}$
and $I_n/\mathfrak{m}_nI_n$, respectively. Let
$$f_0=\prod_{i=1}^{d_1-d_1'}(x_n-x_i)\cdot\prod_{i=d_1-d_1'+1}^{n-1}(y_n-y_i).$$ Then the linear map
$h: M'_{d'_1,d'_2}\to M_{d_1,d_2}$ that maps
$\bar{f}$ to  $\overline{f_0f}$ is injective.
\end{lem}
\begin{proof}
First observe that $\binom{n-1}{2}-d'_1-d'_2=\binom{n}{2}-d_1-d_2$, which we denote by $k$. It is also easy to check that $h$ is well-defined.

We now explain that the following triangle is commutative:
$$\xymatrix{M'_{d'_1,d'_2}\ar[r]^h\ar[rd]_{\bar{\varphi}'}&M_{d_1,d_2}\ar[d]^{\bar{\varphi}}\\ &
\C[\rho]_k}$$
i.e., $\bar{\varphi}'(\bar{f})$ is
identical with $\bar{\varphi}(\overline{f_0f})$. Indeed, for $D'\in\D_{n-1}$ of bidegree $(d_1,d_2)$, let $f=\Delta(D')$, then $f_0f=\Delta(D)$ for $D=D'\cup\{(d_1-d_1',n-1-d_1+d_1')\}\in\D_n$. Then $k(D')=k(D)=k$. Let $A'$ (resp. $A$) be the $(n-1)\times(n-1)$ matrix (resp. $n\times n$ matrix) in the definition of $\bar{\varphi}'(f)$ (resp. $\bar{\varphi}(f_0f)$). Since $A'$ is the first $(n-1)\times(n-1)$ minor of $A$ and the last row of $A$ is $(0,\dots,0,1)$, $\det(A')=\det(A)$.
Therefore $\bar{\varphi}'(\bar{f})=(-1)^k\det(A')=(-1)^k\det(A)=\bar{\varphi}(\overline{f_0f})$.

Now since $\bar{\varphi}':M'_{d'_1,d'_2}\to\C[\rho]_k$ is injective for $k\le (n-1)-3$,
$h$ is also injective.
\end{proof}

\section{Limits of the Modified Algebraic Higher $q,t$-Catalan Numbers}
\label{proof_section}
This section is organized as follows. First we prove Theorem \ref{thm:bound}, which gives a spanning set of the vector space $M^{(m)}_{d_1,d_2}$ for certain $d_1$ and $d_2$. The essential tool is the Transfactor Lemma (Lemma \ref{lem:transfactor}) that allows us to modify staircase forms within the equivalence class modulo lower degrees.   Then we prove Corollary \ref{02102011vv} and find the limits of the modified algebraic higher $q,t$-Catalan numbers.

\begin{lem}\label{lem:k}
Let $D\in\D_n$, let $S$ be a staircase form of $D$,
and let $B(S)$ be the block diagonal form of $S$. Then the number of $1\times
1$ blocks in $B(S)$ is at least $n-2k(D)$.
\end{lem}
\begin{proof}
Suppose the number of size-1 blocks in $B(S)$ is $t$, and the
other blocks have sizes $s_1,\dots,s_r$. On one hand,
$t+\sum_{i=1}^rs_i=n$. On the other hand, a block of size $s_i$ contributes at least $s_i-1$ to $k(D)$, hence
$\sum_{i=1}^r(s_i-1)\le k(D)$. Since $s_i\ge 2$, we have $s_i\le2(s_i-1)$ and
$t=n-\sum_{i=1}^rs_i\ge n-\sum_{i=1}^r2(s_i-1)\ge n-2k(D)$.
\end{proof}

\begin{lem}[Transfactor Lemma {\cite[Lemma 15]{LL}}]\label{lem:transfactor}
Let $D=\{P_1,\dots,P_n\}\in\mathfrak{D}_n$ and $P_i=(a_i,b_i)$ be as in \S2. Let $i,j$ be two integers satisfying $1\le i\neq j\le n$, $|P_i|=i-1$, $|P_{i+1}|=i$, $|P_j|=j-1$, $|P_{j+1}|=j$, $b_i>0$, $a_j>0$   (we define $|P_{n+1}|=n$).
Let $D'$ be obtained from $D$ by moving $P_i$ to southeast and $P_j$ to northwest, i.e.,
$$D'=\big{\{}P_1, \dots, P_{i-1}, P_i+(1,-1),P_{i+1}, \dots, P_{j-1}, P_j+(-1,1),P_{j+1}, \dots, P_n\big{\}}.$$
Then $\Delta(D)\equiv\Delta(D')$ (modulo lower degrees).
\end{lem}

\begin{thm}\label{thm:bound}
Assume $n,m,k,d_1,d_2\in\N$ satisfy $n\geq 3$, $m>0$,
$k=m\binom{n}{2}-d_1-d_2<n/2-1$, and $d_2<n/2-1$.
For each $\mu\in\Par(d_2,k)$, let $S_\mu$ be an arbitrary minimal
staircase form of bidegree $(d_1-(m-1)\binom{n}{2},d_2)$
and partition type $\mu$. Then
$M_{d_1,d_2}^{(m)}$ is generated as a vector space
by $$\left\{(\det S_\mu)\prod_{1\leq i<j\leq
n}(x_i-x_j)^{m-1}\right\}_{\mu\in\Par(d_2,k)}.$$
Consequently, $$\dim M_{d_1,d_2}^{(m)}\leq p(d_2,k).$$
\end{thm}
\begin{proof}
We use induction on $m$. The base case $m=1$ is done in \cite{LL}. Let us briefly sketch a proof for the base case.

For each $\mu\in\Par(d_2,k)$, the assumption $k<n/2 -1$ implies that there exists a minimal
staircase form, say $S_\mu$, of bidegree $(d_1-(m-1)\binom{n}{2},d_2)$
and partition type $\mu$. Let $D_\mu$ be an element in $\D_n$, whose staircase form is $S_\mu$. Since $\bar{\varphi}(\Delta(D_\mu))=\rho_\mu$ and $\bar{\varphi}$ is injective in this case, $M_{d_1,d_2}$ is generated by $\Delta(D_\mu)
(\equiv\det S_\mu)$.

Now assume that $m\geq 2$. Note that
$M_{d_1,d_2}^{(m)}$ is generated by products
$\prod_{i=1}^m \Delta(D_i)$, where $D_i\in \D_n$,
$\sum_{i=1}^md_1(D_i)=d_1$ and $\sum_{i=1}^md_2(D_i)=d_2$.
So we only need to prove that each such product is a
linear combination of
$\{\det(S_\mu)\prod_{1\leq i<j\leq n}(x_i-x_j)^{m-1}\}_{\mu\in\Par(d_2,k)}$
modulo lower degrees. Define
$k'=k(D_1)$,  $d_2'=d_2(D_1)$,
$k''=k-k'$, $d_2''=d_2-d_2'$.
By inductive assumption, $\prod_{i=2}^m \Delta(D_i)$ is a linear combination of
 $\{\det(S''_\lambda)\prod_{1\leq i<j\leq n}
(x_i-x_j)^{m-2}\}_{\lambda\in\Par(d_2'',k'')}$ modulo lower degrees, and
$\Delta(D_1)$ is a linear combination of
$\{\det(S'_\nu)\}_{\nu\in\Par(d_2',k')}$ modulo lower degrees. Hence
$\prod_{i=1}^m \Delta(D_i)$ is a linear combination of $\{
\det(S'_\nu)\det(S''_\lambda)\prod_{1\leq i<j\leq n}(x_i-x_j)^{m-2}\}_{\nu\in
\Par(d_2',k'),\,\,\lambda\in\Par(d_2'',k'')}$ modulo lower degrees. So to prove the theorem, it suffices to show the
following statement:

\noindent $(*)$\quad $\det(S'_\nu)\det(S''_\lambda)$ is a linear combination
of $\{\det(S_\mu)\prod_{1\leq i<j\leq n}(x_i-x_j)\}_{\mu\in\Par(d_2,k)}$
modulo lower degrees.

Since $S'_\nu$ and $S''_\lambda$ can be arbitrary minimal
staircase forms of fixed bidegree and fixed partition type, we may assume that all the $1\times 1$ blocks but the first
one in the block diagonal form $B(S'_\nu)$ are below bigger blocks, and that all the $1\times
1$ blocks in the block diagonal form $B(S''_\lambda)$ are above bigger blocks. Let $T'$
(resp. $T''$) be the product of determinants of the blocks of size greater than 1 in the block diagonal form $B(S'_\nu)$
(resp. $B(S''_\lambda)$). We have
$$
\det(S'_\nu)=T'\prod_{j=a}^n \prod_{i=1}^{j-1} z^{(1)}_{ij},
\quad
\det(S''_\lambda)=\left(\prod_{j=2}^b \prod_{i=1}^{j-1}z^{(2)}_{ij}\right)T'',
$$
where $z^{(t)}_{ij}=x_i-x_j$ or $y_i-y_j$ for $t=1,2$.
The numbers of size-1 blocks
in $B(S'_\nu)$ and $B(S''_\lambda)$ are $n-a+2$ and $b$, respectively.
We assume without loss of generality that $S'_\nu$
has no more size-1 blocks than $S''_\lambda$, in other words,
that $n-a+2\le b$.
By Lemma~\ref{lem:k}, $n-a+2\ge n-2k'$ and $b\ge n-2k''$. Then
$$2b\ge(n-a+2)+b\ge 2n-2k'-2k'',$$
therefore $b-a\ge n-2-2k>n-2-(n-2)=0$ and $b\ge n-k>n/2+1$.
Since $d_2<n/2-1\le b-1$, we can use Lemma~\ref{lem:transfactor}
to adjust the first $b$ columns in $S''_\lambda$
without changing $\det(S''_\lambda)$ (modulo lower degrees),
so that $z^{(2)}_{ij}=x_i-x_j$ for $1\leq i<j\leq b-1$. Note that $z^{(2)}_{ib}$ can be either $x_i-x_j$ or $y_i-y_j$ for $1\leq i<b$.
Similarly, we can adjust the last $n-b+2$ columns in $S'_\nu$ such that
$z^{(1)}_{ij}=x_i-x_j$ for $b\leq j\leq n$ and  $1\leq i<j$. Then
$$
\det(S'_\nu)=T'\prod_{j=a}^{b-1}\prod_{i=1}^{j-1} z^{(1)}_{ij}
 \prod_{j=b}^n \prod_{i=1}^{j-1} (x_i-x_j),
\quad
\det(S''_\lambda)=\left(\prod_{j=2}^{b-1} \prod_{i=1}^{j-1}(x_i-x_j)\right)
\left(\prod_{i=1}^{b-1} z^{(2)}_{ib}\right)T'',
$$
and
$$\det(S'_\nu)\det(S''_\lambda)=A\prod_{1\leq i<j\leq n}(x_i-x_j),
\quad \textrm{where }
A=T'\left(\prod_{j=a}^{b-1}\prod_{i=1}^{j-1} z^{(1)}_{ij}\right)
\left(\prod_{i=1}^{b-1} z^{(2)}_{ib}\right)T''.$$
One verifies that
$A$ is a polynomial of bidegree $(d_1-(m-1)\binom{n}{2},d_2)$ in $I$.
Applying the base case $m=1$, we conclude
that $\det(S'_\nu)\det(S''_\lambda)$ is a linear combination of
$$\left\{\det(S_\mu)\prod_{1\leq i<j\leq n}(x_i-x_j)
\right\}_{\mu\in\Par(d_2,k)}$$
modulo lower degrees. This proves $(*)$.
\end{proof}

The following lemma about partition numbers is needed in the proof of Corollary \ref{02102011vv}.
\begin{lem}\label{lem:sum p} Let $a$ be a positive integer.
Then $\sum_{i=0}^a p(i,a-i)=p(a)$.
\end{lem}
\begin{proof}
Given a partition $\nu=(\nu_1,\dots,\nu_\ell)$ of $a$ satisfying
$\nu_1\le\cdots\le\nu_\ell$, we let $i=\nu_\ell$, and
send $\nu$ to the transpose of the partition $(\nu_1,\dots,\nu_{\ell-1})$,
which is a partition of $a-i$ into at most $i$ parts. This gives a one-to-one
correspondence from $\Par(a)$ to $\bigcup_{i=0}^a \Par(i,a-i)$. Counting the
cardinalities of the two sets gives the stated equality.
\end{proof}

Now we are ready to prove the following consequence of Theorem \ref{thm:bound}, and thus complete the proof of Theorem \ref{mainthm}.
\begin{cor}\label{02102011vv}
Let $n,m,k,d_2\in\N$ satisfy $n\ge 3$, $m>0$, and $k+d_2<n/2-1$. Define
$d_1=m\binom{n}{2}-k-d_2$.
Then the coefficient of $q^{d_1}t^{d_2}$ in $AC^{(m)}_n(q,t)$ is
$$\dim M^{(m)}_{d_1, d_2}=p(d_2,k).$$
As a consequence, 
$$\lim_{n\to\infty} q^{m\binom{n}{2}}AC_n^{(m)}(q^{-1},t)
=\prod_{i=1}^\infty (1-tq^i)^{-1}
  =\sum_{\mu\in\Par} q^{\area(\mu)}t^{\ell(\mu)}.$$
\end{cor}
\begin{proof}
We recalled in the introduction that $AC^{(m)}_n(q,1)=RC^{(m)}_n(q,1)
=\sum_{\pi\in\mathcal{D}^{(m)}_n} q^{\area(\pi)}$.
Then for each $d_1$, $\sum_{d_2=0}^\infty \dim M_{d_1,d_2}^{(m)}$ is the
number of $m$-Dyck paths $\pi\in\mathcal{D}^{(m)}_n$ with
$\area(\pi)=d_1$. Each such $m$-Dyck path uniquely determines a Ferrers
diagram of size $m\binom{n}{2}-d_1$ consisting of the
set of boxes above the $m$-Dyck path in the $mn\times n$ triangle.
On the other hand, since any Ferrers diagram of size less than $n$
determines an
$m$-Dyck path and $m\binom{n}{2}-d_1=k+d_2<n$, each Ferrers diagram of size
$m\binom{n}{2}-d_1$ determines an $m$-Dyck path in $\mathcal{D}^{(m)}_n$.
Therefore the number of such $m$-Dyck paths is equal to the partition number
$p\left(m\binom{n}{2}-d_1\right)$, and
$$\sum_{d_2=0}^{m\binom{n}{2}-d_1} \dim
M_{d_1,d_2}^{(m)}=p\left(m\binom{n}{2}-d_1\right)
=\sum_{d_2=0}^{m\binom{n}{2}-d_1}
p\left(d_2,m\binom{n}{2}-d_1-d_2\right),$$
where the second equality is because of Lemma~\ref{lem:sum p}.
On the other hand, Theorem~\ref{thm:bound} asserts that
$$\dim M_{d_1,d_2}^{(m)}\le p\left(d_2,m\binom{n}{2}-d_1-d_2\right).$$
Therefore each inequality is actually an equality. This implies
$\dim M^{(m)}_{d_1, d_2}=p(d_2,k)$.

To prove the consequence, note that for any fixed nonnegative integers $k,h$, whenever $n>2(k+h+1)$,
the coefficient of $q^{m\binom{n}{2}-k-h}t^h$ in
$AC_n^{(m)}(q,t)$ is equal to $p(h,k)$. Therefore
$$\aligned
\lim_{n\to\infty} q^{m\binom{n}{2}}AC_n^{(m)}(q^{-1},t)
&=\sum_{k,h\ge0}p(h,k)q^{m\binom{n}{2}-(m\binom{n}{2}-k-h)}t^h\\
&=\sum_{k,h\ge0}p(h,k)q^{k+h}t^h
=\prod_{i=1}^\infty (1-tq^i)^{-1}. 
\endaligned
$$
The second equality of the consequence is because of Theorem 
\ref{thm:LW}.
\end{proof}

\section{Comparison of Coefficients of $AC^{(m)}_n(q,t)$ to Partition Numbers}
This section proves Theorem \ref{thm:bound2} by showing the two inequalities $\dim M_{d_1,d_2}^{(m)}\le p(d_2,k)$ and $\dim M_{d_1,d_2}^{(m)}\ge p(d_2,k)$ separately. 
We first use Grafting Lemma (Lemma \ref{lem:grafting}) to prove the Higher Transfactor Lemma (Lemma \ref{lem:higher transfactor}), then use both lemmas to prove Lemma \ref{lem:mm}, which plays a key role in the proof of the former inequality. 
Finally, we complete the proof of Theorem \ref{thm:bound2} by showing
the latter inequality using results from section 3 and 4.

\begin{lem}[Grafting Lemma]\label{lem:grafting}
Let $D_1=\{P_1,\dots,P_n\}$ and $D_2=\{Q_1,\dots,Q_n\}$ be in $\D_n$,
where the $P_i$ and $Q_i$ are listed in increasing graded lexicographic order.
Suppose
$|P_r|=|Q_r|=r-1$. Let $D'_1=\{P_1,\dots,P_{r-1},Q_r,\dots,Q_n\}$ and
$D'_2=\{Q_1,\dots,Q_{r-1},P_r,\dots,P_n\}$. Then
$$\Delta(D_1)\cdot\Delta(D_2)\equiv\Delta(D'_1)\cdot\Delta(D'_2)$$ in $M^{(2)}$.\qed
\end{lem}

This can be obtained by switching blocks in block diagonal forms of $D_1$ and $D_2$, so we omit the detailed proof of the lemma. The following example illustrates the
idea of the proof.

\begin{exmp} Consider $D_1$, $D_2$, $D_1'$, and $D_2'$ pictured below.
 $$%
    \setlength{\unitlength}{1.2pt}
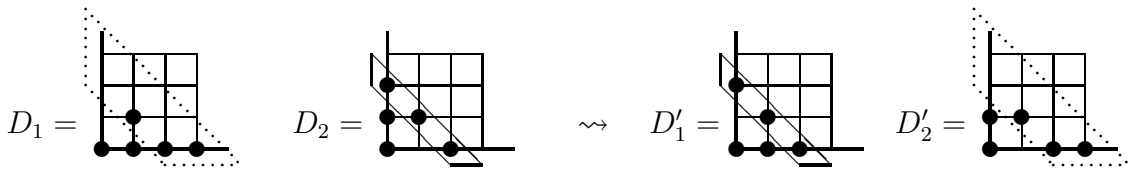
    \begin{picture}(60,50)(-20,0)
    \put(-20,5){$D_1=$}
    \put(10,0){\circle*{5}}\put(20,0){\circle*{5}}
    \put(20,10){\circle*{5}}\put(30,0){\circle*{5}}\put(40,0){\circle*{5}}
    \boxs{10,0}\boxs{20,0}\boxs{30,0}
    \boxs{10,10}\boxs{20,10}\boxs{30,10}
    \boxs{10,20}\boxs{20,20}\boxs{30,20}
    \linethickness{1pt}\put(10,0){\line(0,1){37}}
    \linethickness{1pt}\put(10,0){\line(1,0){40}}
    \multiput(5,20)(2,-2){13}{\circle*{.9}}
    \multiput(7,42)(2,-2){24}{\circle*{.9}}
    \multiput(5,23)(0,3){8}{\circle*{.9}}
    \multiput(30,-5)(3,0){8}{\circle*{.9}}
    \end{picture}
    \setlength{\unitlength}{1.2pt}
    \begin{picture}(110,50)(-50,0)
    \put(-20,5){$D_2=$}
    \put(10,0){\circle*{5}}\put(10,10){\circle*{5}}
    \put(10,20){\circle*{5}}\put(20,10){\circle*{5}}\put(30,0){\circle*{5}}
    \boxs{10,0}\boxs{20,0}\boxs{30,0}
    \boxs{10,10}\boxs{20,10}\boxs{30,10}
    \boxs{10,20}\boxs{20,20}\boxs{30,20}
    \linethickness{1pt}\put(10,0){\line(0,1){37}}
    \linethickness{1pt}\put(10,0){\line(1,0){40}}
    \put(5,20){\line(1,-1){25}}
    \put(5,30){\line(1,-1){35}}
    \put(5,20){\line(0,1){10}}
    \put(30,-5){\line(1,0){10}}
    \end{picture}
    \setlength{\unitlength}{1.2pt}
    \begin{picture}(80,50)(-50,0)
    \put(-40,5){$\leadsto\quad D_1'=$}
    \put(10,0){\circle*{5}}\put(20,0){\circle*{5}}
    \put(10,20){\circle*{5}}\put(20,10){\circle*{5}}\put(30,0){\circle*{5}}
    \boxs{10,0}\boxs{20,0}\boxs{30,0}
    \boxs{10,10}\boxs{20,10}\boxs{30,10}
    \boxs{10,20}\boxs{20,20}\boxs{30,20}
    \linethickness{1pt}\put(10,0){\line(0,1){37}}
    \linethickness{1pt}\put(10,0){\line(1,0){40}}
    \put(5,20){\line(1,-1){25}}
    \put(5,30){\line(1,-1){35}}
    \put(5,20){\line(0,1){10}}
    \put(30,-5){\line(1,0){10}}
    \end{picture}
    \setlength{\unitlength}{1.2pt}
    \begin{picture}(80,50)(-50,0)
    \put(-20,5){$D_2'=$}
    \put(10,0){\circle*{5}}\put(10,10){\circle*{5}}
    \put(20,10){\circle*{5}}\put(30,0){\circle*{5}}\put(40,0){\circle*{5}}
    \boxs{10,0}\boxs{20,0}\boxs{30,0}
    \boxs{10,10}\boxs{20,10}\boxs{30,10}
    \boxs{10,20}\boxs{20,20}\boxs{30,20}
    \linethickness{1pt}\put(10,0){\line(0,1){37}}
    \linethickness{1pt}\put(10,0){\line(1,0){40}}
    \multiput(5,20)(2,-2){13}{\circle*{.9}}
    \multiput(7,42)(2,-2){24}{\circle*{.9}}
    \multiput(5,23)(0,3){8}{\circle*{.9}}
    \multiput(30,-5)(3,0){8}{\circle*{.9}}
    \end{picture}
$$
Then
{\small
$$
\Delta(D_1)\cdot\Delta(D_2)=
\det\begin{vmatrix}
 1&0&0&0&0\\
 0&x_{21}&0&0&0\\
 0&0&x_{31}y_{32}& x_{31}x_{32}&0\\
 0&0&x_{41}y_{42}& x_{41}x_{42}&x_{41}x_{42}x_{43}\\
 0&0&x_{51}y_{52}& x_{51}x_{52}&x_{51}x_{52}x_{53}\\
\end{vmatrix}
\cdot
\det\begin{vmatrix}
 1&0&0&0&0\\
 0&y_{21}&0&0&0\\
 0&0&y_{31}y_{32}& x_{31}y_{32}&x_{31}x_{32}\\
 0&0&y_{41}y_{42}& x_{41}y_{42}&x_{41}x_{42}\\
 0&0&y_{51}y_{52}& x_{51}y_{52}&x_{51}x_{52}\\
\end{vmatrix},
$$
}
and
{\small
$$
\Delta(D'_1)\cdot\Delta(D'_2)=
\det\begin{vmatrix}
 1&0&0&0&0\\
 0&x_{21}&0&0&0\\
 0&0&y_{31}y_{32}& x_{31}y_{32}&x_{31}x_{32}\\
 0&0&y_{41}y_{42}& x_{41}y_{42}&x_{41}x_{42}\\
 0&0&y_{51}y_{52}& x_{51}y_{52}&x_{51}x_{52}\\
\end{vmatrix}
\cdot
\det\begin{vmatrix}
 1&0&0&0&0\\
 0&y_{21}&0&0&0\\
 0&0&x_{31}y_{32}& x_{31}x_{32}&0\\
 0&0&x_{41}y_{42}& x_{41}x_{42}&x_{41}x_{42}x_{43}\\
 0&0&x_{51}y_{52}& x_{51}x_{52}&x_{51}x_{52}x_{53}\\
\end{vmatrix}.
$$
}
One readily verifies that the two products are equal.
\end{exmp}

\begin{defn} For $k\le n-4$ and $d_1+d_2+k=2\binom{n}{2}$,
define a subspace $N_{d_1,d_2}$ of $M^{(2)}_{d_1,d_2}$ by
$$N_{d_1,d_2}=\left\{\begin{array}{ll}
        M_{d_1-\binom{n}{2},d_2}\cdot f_{\binom{n}{2},0} & \textrm{ if $d_2\le k$}; \\
        M_{d_1,d_2-\binom{n}{2}}\cdot f_{0,\binom{n}{2}} & \textrm{ if $d_1\le k$}; \\
        M_{d_1+d_2-\binom{n}{2}-k,k}\cdot f_{\binom{n}{2}-d_2+k,d_2-k} & \textrm{  otherwise}.
      \end{array}
\right.
$$
\end{defn}

\begin{lem}[Higher Transfactor Lemma]\label{lem:higher transfactor}
Suppose $k\le n-4$, $d'_1\le d_1$, $d'_2\le d_2$, $d_1+d_2+k=2\binom{n}{2}$,
and $d'_1+d'_2+\binom{n}{2}=d_1+d_2$.

{\rm(i)}
If $d'_2<d_2$ and $d'_1\ge k+1$, then
 $$M_{d'_1,d'_2}\cdot f_{d_1-d'_1, d_2-d'_2}\subseteq M_{d'_1-1,d'_2+1}\cdot f_{d_1-d'_1+1, d_2-d'_2-1}$$
as subspaces of  $M^{(2)}_{d_1,d_2}$.

{\rm(ii)}
If $d'_1< d_1$ and $d'_2\ge k+1$, then
 $$M_{d'_1,d'_2}\cdot f_{d_1-d'_1, d_2-d'_2}\subseteq M_{d'_1+1,d'_2-1}\cdot f_{d_1-d'_1-1, d_2-d'_2+1}$$
as subspaces of  $M^{(2)}_{d_1,d_2}$.

{\rm(iii)}  $M_{d'_1,d'_2}\cdot f_{d_1-d'_1,d_2-d'_2}$ is a subspace
of $N_{d_1,d_2}$.
Moreover, if  $d'_1,d'_2\ge k$, then
$M_{d'_1,d'_2}\cdot f_{d_1-d'_1,d_2-d'_2}$ is equal to $N_{d_1,d_2}$.
\end{lem}

\begin{proof}
(i)
Let $$
P_n=\left\{\begin{array}{ll}
        (n-1,0), & \textrm{ if $d'_2<k$}; \\
        (n-1-d'_2+k,d'_2-k), & \textrm{ if $k\le d'_2\le n-2+k$}; \\
        (1,n-2), & \textrm{  if $d'_2>n-2+k$}.
\end{array}
\right.
$$
Then there exists a basis $\{\Delta(D_i)\}$
of $M_{d_1,d_2}$ such that the last point of each $D_i$ is $P_n$.
Indeed, consider the first case $d'_2<k$. Let $M'_{d'_1-(n-1),d'_2}$ be the
indicated graded piece of $I_{n-1}/\mathfrak{m}_{n-1}I_{n-1}$.
Let $\{\Delta(D'_i)\}$ be a basis of $M'_{d'_1-(n-1),d'_2}$, and let $D_i$
be obtained from $D'_i$ by adding the point $P_n$. Since
$M'_{d''_1,d'_2}$ and $M_{d_1,d_2}$ have the same dimension $p(d_2,k)$,
 Lemma \ref{lem:injective} implies that $\{\Delta(D_i)\}$ forms a basis of $M_{d_1,d_2}$. The other two cases can be proved similarly.

Now for each $D_i=\{P_1,\dots,P_n\}$,
define $D'_i=\{P_1,\dots,P_{n-1},P_n+(-1,1)\}$.
By the Grafting Lemma~\ref{lem:grafting}, we have
$\Delta(D_i)\cdot f_{d_1-d'_1,d_2-d'_2}\equiv\Delta(D'_i)\cdot f_{d_1-d'_1+1,d_2-d'_2-1}$ in $M^{(2)}_{d_1,d_2}$.
Then the inclusion stated in (i) follows immediately.

(ii) This is symmetric to (i).

(iii) This follows from (i) and (ii).
\end{proof}

\begin{lem}\label{lem:mm}
Assume $n,d'_1,d'_2,k',d''_1,d''_2,k''\in\N$ satisfy $n\geq 6$,
$k'=\binom{n}{2}-d'_1-d'_2$, $k''=\binom{n}{2}-d''_1-d''_2$,
$k'+k''\le n-6$, and $(d'_1,d'_2)+(d''_1,d''_2)=(d_1,d_2)$. Then
 $$M_{d'_1,d'_2}\cdot M_{d''_1,d''_2}\subseteq N_{d_1,d_2}$$
as subspaces of $M^{(2)}_{d_1,d_2}$.
\end{lem}

\begin{proof}
Define $n'=k'+3$. First, we claim that
$M_{d'_1,d'_2}$ has a basis
consisting of elements of the form
$$\Delta(D')=\Delta(\{P'_1,\dots,P'_n\}), \textrm{ where }
|P'_i|=i-1 \textrm{ for } n'+1\le i\le n,$$
and $M_{d''_1,d''_2}$ has a basis
consisting of elements of the form
$$\Delta(D'')=\Delta(\{P''_1,\dots,P''_n\}),
\textrm{ where } |P'_i|=i-1 \textrm{ for } 0\le i\le n'.$$
Indeed, one may find a pair of integers $(e'_1,e'_2)$ such that
 $$\min(d'_1,k')\le e'_1\le d'_1\le e'_1+\binom{n}{2}-\binom{n'}{2},$$
 $$\min(d'_2,k')\le e'_2\le d'_2\le e'_2+\binom{n}{2}-\binom{n'}{2},$$
 $$\mbox{and } e'_1+e'_2=\binom{n'}{2}-k'.$$
Then we choose $P'_{n'+1},\dots,P'_n$ such that $|P'_i|=i-1$ for
$n'+1\le i\le n$, and the sum of their bidegrees is $(d'_1-e'_1,d'_2-e'_2)$.
Choose a basis $\{\Delta(\tilde{D}')\}$
of $(I_{n'}/\mathfrak{m}_{n'}I_{n'})_{e'_1,e'_2}$
and replace each $\tilde{D}'=\{Q_1,\dots,Q_{n'}\}$ by
$$D'=\{Q_1,\dots,Q_{n'},P'_{n'+1},\dots,P'_n\}.$$
In this way, we obtain a basis for $M_{d'_1,d'_2}$. On the other hand,
one can verify that there exist a pair of integers $(e''_1,e''_2)$
and a nonnegative integer $c\le n'$  such that
 $$\min(d''_1,k'')\le e''_1\le d''_1-(n-n')c\le e''_1+\binom{n'}{2},$$
 $$\min(d''_2,k'')\le e''_2\le d''_2-(n-n')(n'-c)\le e''_2+\binom{n'}{2},$$
 $$\mbox{and } e''_1+e''_2=\binom{n-n'}{2}-k''.$$
Then we choose $P''_1,\dots,P''_{n'}$ such that $|P'_i|=i-1$
for $1\le i\le n'$, and the sum of their bidegrees is
$(d''_1-(n-n')c-e''_1,d''_2-(n-n')(n'-c)-e''_2)$. Take a
basis $\{\Delta(\tilde{D}'')\}$ of
$(I_{n-n'}/\mathfrak{m}_{n-n'}I_{n-n'})_{e''_1,e''_2}$,
and replace each $\tilde{D}''=\{Q_1,\dots,Q_{n-n'}\}$ by
$$D''=\big{\{}P''_1,\dots, P''_{n'}, Q_1+(c,n'-c),Q_2+(c,n'-c),\dots,Q_{n-n'}
+(c,n'-c)\big{\}}.$$
In this way, we obtain a basis for $M_{d''_1,d''_2}$.

Next, using the Grafting Lemma~\ref{lem:grafting},
$$\Delta(D')\Delta(D'')\equiv
\Delta(\{P'_1,\dots,P'_{n'},P''_{n'+1},\dots,P''_n\})\Delta(\{P''_1,\dots,
P''_{n'},P'_{n'+1},\dots,P'_n\}),$$
hence is in $N_{d_1,d_2}$ by Lemma \ref{lem:higher transfactor}(iii).
\end{proof}

\begin{proof}[Proof of Theorem \ref{thm:bound2}]
Without loss of generality, we assume $d_1\ge d_2$. After applying Lemma \ref{lem:mm} successively, we can conclude that $$M_{d_1,d_2}^{(m)}=M_{d_1-a,d_2-b}\cdot g_{a,b}$$ for some nonnegative integers $a, b$, where $a+b=(m-1)\binom{n}{2}$, and $g_{a,b}=\prod_{i=1}^{m-1} f_{a_i,b_i}$ has bidegree $(a,b)$. Moreover, by inspecting the proof of Lemma \ref{lem:mm} carefully, we can assume $b=\max(0, d_2-k)$.
Therefore
$$\dim M_{d_1,d_2}^{(m)}=\dim (M_{d_1-a,d_2-b}\cdot g_{a,b})\le \dim M_{{d_1-a,d_2-b}}\le p(d_2,k),$$
where the last inequality is because of Theorem \ref{thm:LL}. 

Now we prove $\dim M_{d_1,d_2}^{(m)}\ge p(d_2,k)$.
Take a sufficiently large integer $\tilde{n}>n$ such that
$k, d_2<\tilde{n}/2-1$. Let $\tilde{M}$ be
$I_{\tilde{n}}/\mathfrak{m}_{\tilde{n}}I_{\tilde{n}}$. Let
$$\tilde{f}_0=\prod_{j=n+1}^{\tilde{n}}\prod_{i=1}^{j-1}(x_j-x_i).$$
Define $\tilde{d}_1=d_1+(n+\tilde{n}-1)(\tilde{n}-n)/2$. By applying Lemma \ref{lem:injective} successively,
we conclude that the linear map $h: M_{d_1-a,d_2-b}\to \tilde{M}_{\tilde{d}_1-a,d_2-b}$ that sends
$f$ to  $f\cdot \tilde{f}_0$ is injective. Moreover, since $k\le n-6$,
the domain and the codomain of $h$ have the same dimension $p_{d_2,k}$.
So $h$ is also surjective. Consider the following commutative diagram:
$$
\xymatrix{
M_{d_1-a,d_2-b}\ar[r]^h\ar[d]_{\psi_1}&\tilde{M}_{\tilde{d}_1-a,d_2-b}\ar[d]^{\tilde{\psi}_1}\\
M_{d_1-a,d_2-b}\cdot g_{a,b}\ar[r]\ar[d]_{\psi_2}&\tilde{M}_{\tilde{d}_1-a,d_2-b}\cdot\tilde{g}_{a,b}\ar[d]^{\tilde{\psi}_2}\\
M^{(m)}_{d_1,d_2}\ar[r]&\tilde{M}^{(m)}_{\tilde{d}_1,d_2}
}$$
where $\tilde{g}_{a,b}=g_{a,b}\cdot(\tilde{f}_0)^{m-1}$, $\psi_1(f)=f\cdot g_{a,b}$, $\tilde{\psi}_1(f)=f\cdot \tilde{g}_{a,b}$, and both the middle and bottom horizontal maps are given by $f\mapsto f\cdot (\tilde{f}_0)^m$.
Since $h$ and $\tilde{\psi_1}$ are surjective and $\tilde{\psi}_2$ is
an isomorphism, the bottom horizontal map is surjective.
By Corollary~\ref{02102011vv},
$$\dim M^{(m)}_{d_1,d_2}\ge \dim\tilde{M}^{(m)}_{\tilde{d}_1,d_2}=p(d_2,k).$$ Thus the theorem is proved.
\end{proof}

In fact, we expect a stronger statement to hold:
\begin{conj}
Let $n\ge2, m\ge2, d_1, d_2, k$ be positive integers such that $k=m\binom{n}{2}-d_1-d_2$.
Define $\delta=\min(d_1,d_2)$. Then $\dim M^{(m)}_{d_1,d_2}\le p(\delta,k)$.
Moreover, equality holds if and only if $k\leq n-2$.
\end{conj}

\section{Conjectures}

\begin{conj}
For $\pi \in \mathcal{D}_n^{(m)}$ and $1\leq i \leq mn$,
let $a_i(\pi)$ be the number of full squares in the $i$'th column
below $\pi$ and above the line $my=x$, and let $b_{i}(\pi)$ be the
number of full squares $w$ in the $i$'th column which are above $\pi$
and satisfy
$$ m\cdot l(w)  \leq a(w)  \leq  m(l(w)+1).  $$
For $\pi \in \mathcal{D}_n^{(m)}$ and $1\leq j\leq m$, let
 $$D_j(\pi)=\{(a_j(\pi), b_{j}(\pi)),  (a_{j+m}(\pi), b_{j+m}(\pi)),  \ldots,
(a_{j+m(n-1)}(\pi), b_{j+m(n-1)}(\pi))  \}
\subset \mathbb{N}\times \mathbb{N}.$$
Then $\{\prod_{j=1}^m \det(D_j(\pi)):\pi \in \mathcal{D}_n^{(m)}\}$ generates
the $m$-th power $I_n^{m}$ of the ideal $I_n$ generated by
alternating polynomials in $\C[x_1,y_1,\ldots,x_n,y_n]$.
\end{conj}

Note that this conjecture implies $(\ref{eq:comb-eqs})=(\ref{eq:alg-eqs})$. As a matter of fact, not only the generators of $I_n^{(m)}$ but also their syzygies
   have conjecturally nice combinatorial interpretations. For instance, if $m=1$ then we have the following conjecture. A more generalized
   version for $m\geq 1$ will appear elsewhere, as its statement requires a number of definitions including trapezoidal lattice paths in \cite{L0}.

\begin{conj}\label{conj12301}
Let $I_n$ be the ideal generated by alternating polynomials in
$R =\C[\mathbf{x},\mathbf{y}] =\C[x_1,y_1,\ldots,x_n,y_n]$.
Then for each $1\leq i\leq n$, the bigraded Hilbert series of
$$
\emph{Tor}_i(R/I_n,  \C) = \emph{Tor}_i(R/I_n, R/\mathfrak{m})
$$
is equal to
$$
(-1)^{i-1}\sum_{\begin{array}{c}\lambda\vdash n  \\
\emph{spin}(\lambda')=i-1  \end{array}} \langle (s_1)^n, s_{\lambda} \rangle
\langle \nabla(s_{\lambda}  ), s_{(1^n)}\rangle.
$$
(Recall that $\langle (s_1)^n, s_{\lambda}\rangle=f^{\lambda}$,
the number of standard Young tableaux of shape $\lambda$. For definition of spin, see p.6 in \cite{LW2}.)
\end{conj}

This conjecture is verified for $n\leq 6$.  As a special case, we have:
 \begin{conj}\label{conj12302}
The bigraded  Hilbert series of $I_n$ is
$$
\frac{1}{(1-q)^n(1-t)^n}\langle \nabla(s_1^n), s_{(1^n)}\rangle.
$$
\end{conj}

Conjecture~\ref{conj12302} follows from Conjecture~\ref{conj12301}, because $s_1^n=\sum_{\lambda\vdash n}\langle (s_1)^n, s_{\lambda} \rangle s_{\lambda}$.

\section{Appendix: Comparison of Definitions of Higher $q,t$-Catalan Numbers}
In the Introduction, we gave seven definitions (a)--(g) of
the higher $q,t$-Catalan numbers. Here we explain the known relations
among these definitions.

(a)$\Leftrightarrow$(b): There is an obvious bijection
between partitions $\lambda\in\Par_n^{(m)}$ and
$m$-Dyck words $\gamma\in\Gamma_n^{(m)}$, defined as follows.
Given the partition $\lambda$, embed the diagram of $\lambda$
in an $mn\times n$ triangle as shown in Figures~\ref{fig.1}
and~\ref{fig:bounce}.
For $0\leq i<n$, let $\gamma_i$ be the number of complete squares
to the right of $\lambda$ and to the left of the diagonal
in the $(i+1)$'th row from the bottom. For example, when $m=2$,
$n=5$, and $\lambda=(7,5,4)$, we see from Figure~\ref{fig:bounce}
that the associated $2$-Dyck word is $\gamma=(0,2,0,1,1)$.
It is routine to verify that this process defines a bijection
from $\Par_n^{(m)}$ onto $\Gamma_n^{(m)}$ such that
$\area^c(\lambda)=\area(\gamma)$. It is less routine to prove
that $c_m(\lambda)=\dinv_m(\gamma)$; see~\cite[Lemma 6.3.3]{HHLRU}
for the proof. (Note that what we call $c_m(\lambda)$ is
called $b_m(\lambda)$ in~\cite{HHLRU}.)

(b)$\Leftrightarrow$(c): See \cite[\S2.5]{L}
for a bijection from $\Gamma_n^{(m)}$ to $\mathcal{D}_n^{(m)}$
such that if $\gamma$ maps to $\pi$ under the bijection,
then $\area(\gamma)=b_m(\pi)$ and $\dinv_m(\gamma)=\area(\pi)$.
This proves $WC_n^{(m)}(q,t)=DC_n^{(m)}(q,t)$.
On the other hand, it is an open problem to define a bijection
$\gamma\mapsto\pi$ from
$\Gamma_n^{(m)}$ to $\mathcal{D}_n^{(m)}$ satisfying
$\area(\gamma)=\area(\pi)$ and $\dinv_m(\gamma)=b_m(\pi)$.
This problem is equivalent to proving bijectively that
the combinatorial definitions (a), (b), and (c) are
symmetric in $q$ and $t$.

(d)$\Leftrightarrow$(e): One can use well-known facts about
Macdonald polynomials to prove that $SC^{(m)}_n(q,t) = RC^{(m)}_n(q,t)$
(cf.~\cite{GH} and~\cite{CL}).
Indeed, since $e_n=\sum_{\mu\vdash n} ((1-q)(1-t)B_\mu\Pi_\mu/w_\mu)\tilde{H}_\mu$ and
$\nabla(\tilde{H}_\mu)=T_\mu\tilde{H}_\mu$, linearity of $\nabla$ gives
$\nabla^m(e_n)=\sum_{\mu\vdash n} ((1-q)(1-t)T_\mu^m B_\mu\Pi_\mu/w_\mu)\tilde{H}_\mu$.
Since $\langle\tilde{H}_\mu,e_n\rangle=T_\mu$, we can conclude that
$\langle\nabla^m(e_n),e_n\rangle =
\sum_{\mu\vdash n}(1-q)(1-t)T_\mu^{m+1}B_\mu\Pi_\mu/w_\mu$, as desired.

(d)$\Leftrightarrow$(f):
Let $J$ be the ideal in $\C[\mathbf{x},\mathbf{y}]$
generated by polarized power sums $\sum_{i=1}^n x_i^hy_i^k$ ($h+k\ge 1$).
One can also describe $J$ as the ideal generated by all $S_n$-invariant
polynomials without constant term, where $S_n$ acts diagonally~\cite{H94}.
Let $\varepsilon$ be the sign representation of $S_n$. It is proved in
\cite[Proposition 6.1.1]{HHLRU}
that $$\nabla^m(e_n(z_1,z_2,\ldots))=
\mathcal{F}_{\varepsilon^{m-1}\otimes I^{m-1}/JI^{m-1}}(z_1,z_2,\ldots;q,t),$$
where the right side denotes
the Frobenius series of  $\varepsilon^{m-1}\otimes I^{m-1}/JI^{m-1}$.
(Note that the meanings of $I$ and $J$ are switched
 in~\cite{HHLRU}.)
On the other hand, one may check that the $S_n$-alternating part
$\big{(}\varepsilon^{m-1}\otimes I^{m-1}/JI^{m-1} \big{)}^\varepsilon$
is isomorphic to $\varepsilon^{m-1}\otimes I^{m}/{\mathfrak{m}}I^{m}$.
We can extract the $S_n$-alternating part from the Frobenius series
by taking the scalar product with $e_n=s_{(1^n)}$.  Therefore,
$$\aligned
SC^{(m)}_n(q,t)&=\langle \nabla^m(e_n),e_n\rangle
=\sum_{u,v\geq 0}q^ut^v\dim(\varepsilon^{m-1}\otimes
I^{m}/{\mathfrak{m}}I^{m})_{u,v}
\\ &=\sum_{u,v\geq 0}q^ut^v\dim(I^{m}/{\mathfrak m}I^{m})_{u,v}
=\sum_{u,v\geq 0}q^ut^v\dim M^{(m)}_{u,v}=AC^{(m)}_n(q,t).
\endaligned$$

(e)$\Leftrightarrow$(g): Haiman showed the identity
$$RC^{(m)}_n(q,t)=\sum_{i=0}^{n-1}(-1)^itr_{H^i(Z_n,\mathcal{O}(m))}(q,t)$$
in \cite[\S3, Theorem 2]{H98}.
Then he showed that for $i>0$ and $l\ge 0$,
$H^i(Z_n, P\otimes B^{\otimes l})=0$, where $P$ and $B$
are the vector bundles defined in \cite[\S2]{H02}.
In particular, this implies $H^i(Z_n, \mathcal{O}(k))=0$ for $i>0$
\cite[Introduction and Theorem 2.2]{H02}. Therefore
 $RC^{(m)}_n(q,t)=tr_{H^0(Z_n,\mathcal{O}(m))}(q,t)$,
which is exactly $GC^{(m)}_n(q,t)$.


\begin{thebibliography}{99}

\bibitem{BG}
F. Bergeron, A. Garsia, {\em Science fiction and Macdonald polynomials,}
CRM Proceedings and Lecture Notes AMS \textbf{6} (1999), 363--429.

\bibitem{BP}
F. Bergeron, L.-F. Pr\'eville-Ratelle,
{\em Higher Trivariate Diagonal Harmonics via generalized Tamari Posets,}
arXiv:1105.3738.

\bibitem{BDZ}
N.~Bergeron, F.~Descouens, M.~Zabrocki,
{\em A generalization of $(q,t)$-Catalan and nabla operators,} 20th Annual International Conference on Formal Power Series and Algebraic Combinatorics (FPSAC 2008), 513--527,
Discrete Math. Theor. Comput. Sci. Proc., AJ, Assoc. Discrete Math. Theor. Comput. Sci., Nancy, 2008.


\bibitem{BC}
N.~Bergeron, Z.~Chen, {\em Bases for diagonally alternating harmonic polynomials of low degree}, J. Combin. Theory Ser. A \textbf{118} (2011), no. 1, 37--57.

\bibitem{Bu}
A.~Buryak,
{\em The classes of the quasihomogeneous Hilbert schemes of points on the plane}, arXiv 	arXiv:1011.4459.

\bibitem{CL}
M.~Can, N.~Loehr,
{\em A proof of the q,t-square conjecture}, J. Combin. Theory Ser. A
\textbf{113} (2006), no. 7, 1419--1434.

\bibitem{EHKK}
E.~Egge, J.~Haglund, K.~Killpatrick, D.~Kremer, {\em A Schr\"oder generalization of Haglund's statistic on Catalan paths,}
Electron. J. Combin. 10 (2003), Research Paper 16, 21 pp.

\bibitem{GHag}
A.~Garsia, J.~Haglund,
{\em A proof of the $q,t$-Catalan positivity conjecture,} LaCIM 2000 Conference on Combinatorics, Computer Science and Applications (Montreal, QC). Discrete Math. 256 (2002), no. 3, 677--717.

\bibitem{GH}
A.~Garsia, M.~ Haiman, {\em A remarkable {$q,t$}-{C}atalan sequence and
{$q$}-{L}agrange inversion}, J. Algebraic Combin. \textbf{5}
 (1996), no. 3, 191--244.

\bibitem{GM}
E.~Gorsky, M.~Mazin,
{\em Compactified Jacobians and $q,t$-Catalan Numbers},
arXiv:1105.1151.

\bibitem{Ha}
J.~Haglund,
{\em Conjectured statistics for the q,t-Catalan numbers},
Adv. Math. 175 (2003), no. 2, 319--334.

\bibitem{Hbook}
J. Haglund, \emph{The $q,t$-Catalan Numbers and the Space of Diagonal
Harmonics, with an Appendix on the Combinatorics of Macdonald Polynomials,}
AMS University Lecture Series, 2008.

\bibitem{HHLRU}
J.~Haglund, M.~Haiman, N.~Loehr, J.~Remmel, A.~Ulyanov,
{\em A combinatorial formula for the character
of the diagonal coinvariants}, Duke Math. J. \textbf{126} (2005), no. 2,
195--232.

\bibitem{H03}
M.~Haiman,
{\em Combinatorics, symmetric functions, and Hilbert schemes},
Current developments in mathematics, 2002,
39--111, Int. Press, Somerville, MA, 2003.

\bibitem{H94}
M.~Haiman,
{\em Conjectures on the quotient ring by diagonal invariants,}
J. Algebraic Combin. \textbf{3} (1994), 17--76.

\bibitem{H01}
M.~Haiman, {\em Hilbert schemes, polygraphs and the Macdonald positivity
conjecture}, J. Amer. Math. Soc. \textbf{14} (2001), no. 4,
941--1006.

\bibitem{H98}
M.~Haiman,
{\em $t,q$-Catalan numbers and the Hilbert scheme}, Selected papers in honor of Adriano Garsia (Taormina, 1994). Discrete Math. \textbf{193}
(1998), no. 1-3, 201--224.

\bibitem{H02}
M.~Haiman,
{\em Vanishing theorems and character formulas for the Hilbert scheme of points in the plane},
 Invent. Math. \textbf{149} (2002), no. 2, 371--407.

\bibitem{LL}
K. Lee, L. Li, {\em $q,t$-Catalan numbers and generators for the radical ideal defining the diagonal locus of
$(\C^2)^n$}, Electron. J. Combin. \textbf{18}(1) (2011), Research Paper 158, 34pp.


\bibitem{L0}
N.~Loehr, {\em
Trapezoidal lattice paths and multivariate analogues}, Adv. in Appl. Math. \textbf{31} (2003), no. 4, 597–-629.








\bibitem{L}
N.~Loehr, {\em Conjectured statistics for the higher q,t-Catalan sequences},
Electron. J. Combin. \textbf{12} (2005), Research Paper 9, 54pp.

\bibitem{LW0}
N.~Loehr, G.~Warrington,
{\em Square $q,t$-lattice paths and $\nabla(p_n)$},
Trans. Amer. Math. Soc. 359 (2007), no. 2, 649--669.

\bibitem{LW}
N.~Loehr, G.~Warrington, {\em A continuous family of partition statistics equidistributed with length},
J. Combin. Theory Ser. A \textbf{116} (2009), no. 2, 379--403.

\bibitem{LW2}
N.~Loehr, G.~Warrington, {\em Nested quantum Dyck paths and $\nabla(s_\lambda)$},
Int. Math. Res. Not. (2008), no. 5, Art. ID rnm 157, 29pp.

\bibitem{S}
C.~Stump, {\em $q, t$-Fu\ss-Catalan numbers for finite reflection groups,}
J. Algebraic Combin. 32 (2010), no. 1, 67--97.
\end{thebibliography}
\end{document}